\documentclass[1p]{elsarticleuno}
\usepackage{amsfonts,amsthm,amsmath,amsfonts,latexsym,amssymb}
\usepackage{bbm}
\usepackage{graphicx,upref,calligra}
\usepackage{xcolor}
\usepackage{mathrsfs}
\usepackage{float}
\usepackage[T1]{fontenc}
\usepackage[colorlinks=true,linkcolor=azul,citecolor=azul]{hyperref}

\definecolor{azul}{rgb}{0.1,0.6,0.86}
\newtheorem{fed}{Definition}[section]
\theoremstyle{definition}
\newtheorem{teo}[fed]{Theorem}
\newtheorem*{teo*}{Theorem}
\newtheorem{lem}[fed]{Lemma}
\newtheorem{cor}[fed]{Corollary}
\newtheorem{pro}[fed]{Proposition}
\theoremstyle{definition}
\newtheorem{rem}[fed]{Remark}

\newtheorem{exa}[fed]{Example}

\definecolor{myblue}{rgb}{0,0.33,0.55}
\definecolor{sidebardarkcolor}{rgb}{0.21,0.31,0.40}
\definecolor{sidebarlightcolor}{rgb}{0.7,0.77,0.836}
\def\noi{\noindent}

\def\EOE{\hfill $\blacktriangle$}

\def\bdem{\begin{proof}}
\def\edem{\end{proof}}

\def\eps{\varepsilon}

\def\la{\lambda}
\def\La{\Lambda}

\def\w{\omega}
\def\W{\Omega}
\def\N{\mathbb{N}}
\def\Z{\mathbb{Z}}
\def\R{\mathbb{R}}
\def\C{\mathbb{C}}

\def\D{\mathbb{D}}
\def\T{\mathbb{T}}

\def\Q{\mathbb{Q}}
\def\K{\mathbb{K}}
\def\cD{\mathcal{D}}
\def\ede{\mathcal{D}}
\def\ddd{\mathcal{Q}}

\def\G{\mathcal{G}}

\def\cP{\mathcal{P}}
\def\cQ{\mathcal{Q}}

\def\cT{{\cal T}}

\def\G{\widehat{G}}

\newcommand{\peso}[1]{ \ \mbox{ \rm  #1 } \  }
\DeclareMathOperator{\sgn}{sgn}
\DeclareMathOperator{\leqm}{\preccurlyeq}

\newcommand{\pint}[1]{\displaystyle \left \langle\, #1 \, \right\rangle}
\newcommand{\sub}[2]{{#1}_{\mbox{\tiny{${#2}$}}}}
\newcommand{\dcg}[2]{\mathscr{D}_{#1}^{(#2)}}
\journal{Advances in Mathematics}

\begin{document}
% ----------------------------------------------------------------
% ----------------------------------------------------------------

\begin{frontmatter}

\title{{{Multi-tiling sets, Riesz bases, and sampling near the critical density in LCA groups}}}

\date{}

\author[A1]{Elona Agora}
 \ead{elona.agora@ub.edu}
\author[A1,A2]{Jorge Antezana} 
 \ead{antezana@mate.unlp.edu.ar}
\author[A3,A4]{Carlos Cabrelli\corref{cor1}}
 \ead{cabrelli@dm.uba.ar}
 \ead[url]{http://mate.dm.uba.ar/~cabrelli/}

\address[A1]{Instituto Argentino de Matemática "Alberto P. Calderón" (IAM-CONICET), Argentina} 
%\hfill\textit{elona.agora@ub.edu} 

\address[A2] {Departamento de Matemática, Universidad Nacional de La Plata, Argentina} 
% \hfill  \textit{antezana@mate.unlp.edu.ar}

\address[A3]{Departamento de Matemática, Universidad de Buenos Aires, Argentina}

\address[A4]{Instituto de Matemática "Luis Santaló" (IMAS-CONICET-UBA),  Argentina}
%\hfill \textit{cabrelli@dm.uba.ar}

\cortext[cor1]{Corresponding author: Phone: +5491168610101 - Fax: +541145763335}

%\footnotetext{The research of the authors is partially supported by Grants: CONICET-PIP 435, CONICET-PIP 398, MTM2010-14946, MTM2012-36378,  PICT 2011-436, UBACyT 20020100100638, UNLP-11X585} 

\begin{abstract}

\noi We prove the existence of sampling sets and interpolation sets near the critical density, in Paley Wiener spaces of  a locally compact abelian  (LCA) group $G$. This solves a problem left by  Gr\"ochenig, Kutyniok, and Seip in the article: `Landau's density conditions for LCA groups' (J. of Funct. Anal. 255 (2008) 1831-1850). To achieve this result,  we prove  the existence of universal Riesz bases of characters for $L^2(\Omega)$, provided that the relatively compact subset $\Omega$ of the dual group $\widehat{G}$ satisfies a multi-tiling condition.
This last result generalizes Fuglede's theorem,  and extends to LCA groups setting recent constructions of Riesz bases of exponentials in bounded sets of $\R^d$.
\end{abstract}


%\paragraph{Keywords:}{Sampling; Interpolation; Beurling's densities; Riesz bases; Multi-tiling; Dyadic cubes; Locally compact abelian groups.}

\end{frontmatter}

%\tableofcontents

% ----------------------------------------------------------------
% ----------------------------------------------------------------
\section{Introduction}
% ----------------------------------------------------------------
% ----------------------------------------------------------------

\noi Consider a locally compact abelian (LCA) group  $G$,  and let $\widehat{G}$ denote its dual group. Given a relatively compact Borel subset $\Omega$ of $\widehat{G}$, the space $L^2(\W)$ is identified with the subspace of $L^2(\G)$ consisting of those classes corresponding to functions vanishing almost everywhere on the complement of $\W$. The \textbf{Paley Wiener space} $PW_{\Omega}$ consists of all square integrable functions whose Fourier transform belongs to $L^2(\Omega)$.For this space, a set $\Lambda\subseteq G$ is a \textbf{sampling set} if there exist constants $A, B>0$ such that for any $f\in PW_{\Omega}$,
$$
A||f||_2^2 \leq \sum _{\la \in \La} |f(\la)|^2 \leq B||f||^2_2.
$$
On the other hand, $\La$ is an \textbf{interpolation set} for $PW_{\Omega}$ if the interpolation problem
$$
f(\la)=c_\la, 
$$
has a solution $f\in PW_{\Omega}$ for every  $\{c_\la \}_{\la \in \La}  \in \ell^2(\La)$.  A set $\La$ that is at the same time a sampling and interpolation set is a \textbf{complete interpolation set}. 

\medskip

Using the Fourier transform, it turns out that $\La$ is a sampling set (resp. interpolation set, complete interpolation set) if and only if $\La$, as a set set of characters restricted to $\Omega$, is a frame (resp. Riesz sequence, Riesz basis) of $L^2(\Omega)$.

Sampling and interpolation sets satisfy the following necessary geometric conditions, proved by Landau in \cite{Landau} for $\R^d$, and later on extended to LCA groups in \cite{GKS}:

\begin{enumerate}
\item[(i)] A sampling set $\La$ for $PW_{\Omega}$ satisfies  $\cD^-(\La)\geq m_{\widehat{G}}(\Omega)$;
\item[(ii)] An interpolation set $\La$ for $PW_{\Omega}$ satisfies  $\cD^+(\La)\leq m_{\widehat{G}}(\Omega)$, 
\end{enumerate}
where $m_{\widehat{G}}$ denotes the Haar measure of $\widehat{G}$, and $\cD^+$ and  $\cD^-$ denote the so called upper and lower Beurling's  densities (see Section \ref{densitygroup} for precise definitions). In some sense, the Beurling's densities measure how  the set $\La$ is distributed in $G$ with respect to the distribution of a reference set, that for instance in the case of $\R^d$ is the lattice $\Z^d$.

\medskip

In \cite{GKS}, Gr\"ochenig, Kutyniok and Seip raised the natural question of whether there exist sampling sets and interpolation sets for $PW_{\Omega}$ with densities arbitrarily close to the critical density $m_{\widehat{G}}(\Omega)$. Except for the particular case $G=\R^d$, proved by  Marzo in \cite{JM} adapting a construction of Lyubarskii and Seip~\cite{LS}  and Kohlenberg \cite{K}, the problem remained open till now. The main obstacle, which is a recurrent problem in general LCA groups, is the absence of a natural substitute of rescalings. Therefore, a different approach is required. 

\medskip

The main goal of this paper is to give a complete solution to the aforementioned problem. 
A natural strategy is to show that, given a compact set $\Omega$, there exists an outer (resp. inner) approximation set $\Omega_\eps$ of $\Omega$ such that $L^2(\Omega_\eps)$ has a Riesz basis of characters. As a consequence, we obtain the existence of sampling (resp. interpolation) sets near the critical density (see Theorem \ref{near critical}). Indeed, the Riesz basis of $L^2(\Omega_\eps)$ becomes a frame for $L^2(\Omega)$ if $\Omega_\eps$ is an outer approximation, and it becomes a Riesz sequence if $\Omega_\eps$ is an inner approximation. The problem to accomplish this strategy in general  LCA groups is to prove the existence of such ``good'' approximation sets. 
To overcome this difficulty, we proceed as follows.

 \medskip

 First,  we show that, given a relatively compact Borel set $\Omega\subset \widehat{G}$ that satisfies some tiling condition, the space $L^2(\Omega)$ admits a Riesz basis of characters (see Section \ref{multi} for details). This is motivated by recent results due to Grepstad and Lev in~\cite{GeL} (see also~\cite{Ko}) and it provides an extension of their results. 
In order to prove this generalization we use operator theoretical techniques developed around the theory of shift invariant spaces. The shift invariant techniques provide a better understanding of the problem. As a consequence, besides the extension of the multi-tiling result of Grepstad and Lev to the group setting, we also prove  the converse, that is, if a relatively compact set $\Omega\subset \widehat{G}$ admits a Riesz basis of characters with a periodic set of frequencies, then $\Omega$  satisfies a multi-tiling condition (see Theorem \ref{la vuelta} for details). This is new even for $\R^d$, and complements the result of Grepstad and Lev in order to get a generalization of Fuglede's theorem \cite{F} for Riesz bases and multi-tiling sets. In particular it shows that, if for instance $\Omega$ is a triangle in $\R^2$, then $L^2(\W)$ does not admit a Riesz basis of exponentials with a periodic set of frequencies.
Furthermore, we prove that the boundedness condition over $\W$ cannot be avoided. In order to show that, we construct a counterexample 
that in particular responds in a negative way a question left open by Kolountzakis in \cite{Ko}.

\medskip

In this way we get several candidates for approximation sets, i.e., those sets that satisfy a tiling condition with respect to a lattice. However, in general  there are not sufficiently many of such sets to assure the required approximation, essentially because the group $\widehat{G}$ may not have a rich family of lattices. In order to enlarge the family of candidates, we show that we can also consider sets $\Omega_\eps$ that satisfy the tiling condition in an appropriate quotient group $\G/K$ instead of in the group $\widehat{G}$. Thus, we obtain a Riesz basis of characters in $L^2(\pi(\Omega_\eps))$, where $\pi$ denotes the canonical projection onto that quotient. Finally, we prove that this Riesz basis can be lifted to a Riesz basis of \textit{characters} for $L^2(\Omega_{\eps})$  (see the last part of Section~\ref{Dyadics} for the details).

\bigskip

The paper is organized as follows. In section \ref{prel}, we introduce preliminary results on LCA groups. Section \ref{critical dens} describes the 
main results. In section \ref{multi} we construct the Riesz basis of characters for $L^2(\Omega)$, under a multi-tiling condition on the set $\Omega$. In section  \ref{Dyadics}  we introduce the notion of quasi-dyadic cubes, which are used to construct the approximation sets. Finally, with  all the necessary techniques at hand, we proceed to the proof of the main result in section~\ref{main proof}.

% ----------------------------------------------------------------
% ----------------------------------------------------------------
\section{Preliminaries} \label{prel}
% ----------------------------------------------------------------
% ----------------------------------------------------------------

\noi Throughout this section we  review basic facts on  locally compact abelian groups (for more details see \cite{DE}, \cite{HR1}, \cite{HR2}, \cite{R}), setting in this way the notations we need for the following sections.  Then, we introduce $H$-invariant spaces that generalize the concept of shift invariant spaces in the context of these groups (see  \cite{CP}). 

\subsection{LCA Groups}

\noi Let $G$ denote a Hausdorff locally compact abelian (LCA) group, and $\widehat{G}$ its dual group, that is;
$$
\widehat{G}= \{\gamma: G \to \C, \ \mbox{and}\ \gamma\ \mbox{is a continuous character of}\ G\},
$$
where a character is a function satisfying the following properties:
\begin{itemize}
\item[(i)] $ |\gamma(x)| =1, \,  \forall x\in G$;
\item[(ii)] $\gamma(x+y)=\gamma(x)\gamma(y),  \, \forall x, y \in G$.
\end{itemize}

Thus, the characters generalize the exponential functions $\gamma(x)=\gamma_t(x)= e^{2\pi it x}$ in the case $G=(\R,+)$.  
On every LCA group $G$ there exists a Haar measure. It is a non-negative, regular Borel measure 
$m_G$ that is non-identically zero and translation-invariant, which means:
$$
m_G(E+x)=m_G(E),
$$
for every element $x\in G$ and every Borel set $E\subset G$. This measure is unique up to a constant. Analogously to the Lebesgue spaces, we can define the $L^p(G)=L^p(G, m_G)$ spaces associated to the group $G$ and the measure $m_G$:
$$
L^p(G):= \Big\{f:G \to \C, \,\,\, f \peso{is measurable and} \int_G |f(x)|^p\,dm_G(x)<\infty\Big\}.
$$

\begin{teo} Let $G$ be an LCA group and $\widehat{G}$ its dual. Then
\begin{itemize}
\item[(i)] The dual group $\widehat{G}$, with the operation $(\gamma+\gamma')(x)= \gamma(x)\gamma'(x)$ is an LCA group. The topology in $\widehat{G}$ is the one induced by the identification of the characters of the group with the characters of the algebra $L^1(G)$.
\item[(ii)] The dual group of $\widehat{G}$ is topologically isomorphic to $G$, that is, $\widehat{\widehat{G \,}}\approx G$, with the identification $g\in G \leftrightarrow e_g \in \widehat{\widehat{G\,}}$, where $e_g(\gamma):=\gamma(g)$.
\item[(iii)] $G$ is discrete (resp. compact) if and only if $\widehat{G}$ is compact (resp. discrete).
\end{itemize}
\end{teo}

As a consequence of (ii) of the previous theorem, we could use the notation $(x,\gamma)$ for the complex number $\gamma(x)$, representing either the character $\gamma$ applied to $x$ or the character $x$ applied to $\gamma$.

\bigskip

Taking $f\in L^1(G)$ we define the Fourier transform of $f$, as the function $\hat{f}:\widehat{G}\to\C$ given by
$$
\hat{f}(\gamma)=\int_G f(x)(x,-\gamma)\, dm_G(x), \,\,\, \gamma \in \widehat{G},
$$
If the Haar measure of the dual group $\widehat{G}$ is normalized conveniently,  we obtain the inversion formula
$$
f(x)= \int_{\widehat{G}} \hat{f}(\gamma) (x,\gamma)dm_{\widehat{G}}(\gamma),
$$
for a specific class of functions. In the case that the Haar measures $m_G$ and $m_{\widehat{G}}$ are normalized such that the inversion formula holds, the Fourier transform on $L^1(G)\cap L^2(G)$ can be extended to a unitary operator from $L^2(G)$ onto $L^2(\widehat{G})$. Thus the Parseval formula holds:
$$
\pint{f,g}= \int_G f(x)\overline{g(x)} dm_G(x)
= 
\int_{\widehat{G}} \hat{f}(\gamma)\overline{\hat{g}(\gamma)} dm_{\widehat{G}}(\gamma) = \langle \hat{f},\hat{g}\rangle
$$
for $f,g \in L^2(G)$. We conclude this subsection with the next classical result.

\begin{pro} \label{orthba}
If $G$ is a compact group, then the characters of $G$ form an orthonormal basis for $L^2(G)$.
\end{pro}

\subsection{ H-invariant spaces} 

\noi In this subsection we will review some basic aspects of the theory of shift invariant spaces in LCA groups. We will specially  focus on the Paley Wiener spaces, that constitutes an important family of shift invariant spaces in which we are particularly interested. The reader is referred to  \cite{CP}, where he can find the results in full generality, as well as other results related to shift invariant spaces in LCA groups.  Let $G$ be an LCA group, and let $H$ be a \textbf{uniform lattice} on $G$, i.e.,  a discrete subgroup of $G$ such that $G/ H$ is compact. 
Recall that a \textbf{Borel section} of $G/H$ is a set of representatives of this quotient, that is, a subset $A$ of $G$ containing exactly one element of each coset. Thus, each element $x \in G$ has a unique expression of the form $x = a + h$ with $a \in A$ and $h \in H$. Moreover, it can be proved that there exists a relatively compact Borel section of $G/ H$, which will be called \textbf{fundamental domain} (see \cite{FGg} and \cite{KK}).

\medskip

\begin{fed}
We say that a closed subspace $V\subset L^2(G)$ is $H$-invariant if
$$
f\in V \peso{then} \tau_h f\in V, \,\,\, \forall h \in H,
$$
where $\tau_hf(x)= f(x-h)$.
\end{fed}

\bigskip

As we have mentioned, Paley Wiener spaces are important examples of $H$-invariant spaces, which in this context are defined by
$$
PW_\Omega= \{f\in L^2(G): \,  \, \hat{f} \in  L^2(\Omega)\},
$$
where $\Omega\subset \widehat{G}$ is a Borel set of finite measure (see \cite{GKS}). Actually, this space is invariant by any translation.

\medskip

Let $\La$ be  the \textbf{dual lattice} of $H$; that is, the \textbf{annihilator} of $H$ defined by 
$$ 
\La=\{\gamma\in \widehat{G}: (h, \gamma)=1, \, \text{for all} \, h \in H\}.
$$ 
Suppose that $\Omega$ \textbf{tiles} $\widehat{G}$ by translations of $\La$, i.e.
$$
\Delta_{\Omega}(x):= \sum_{\la\in \La} \chi_{\Omega} (x-\la)=1,  \peso{a.e.}
$$ 
In this case, it is well known that $\{e_h\}_{h\in H}$ is an orthonormal basis of $L^2(\Omega)$. Indeed, since $\La$ is also a uniform lattice, in particular,   $\widehat{G} /  \La$ is compact. So, as we recall in Proposition~\ref{orthba}, $H\simeq (\widehat{G} / \La)^{\widehat{\ }}$ is an orthonormal basis of $L^2(\widehat{G} / \La)$. On the other hand, this space is isometrically isomorphic to $L^2(\Omega)$, because $1$-tiling sets are Borel sections of the quotient group $\widehat{G} / \La$ up to a zero measure set.

\bigskip

A set $\Omega$ \textbf{multi-tiles}, or more precisely \textbf{$k$-tiles} $\widehat{G}$ by translations of $\La$  if
$$
\Delta_{\Omega}(x):= \sum_{\la\in \La} \chi_{\Omega} (x-\la)=k,  \peso{a.e.}
$$ 
For example, if $\Omega$ is a disjoint union of 1-tiling sets then the previous condition is satisfied.  Next lemma shows that the reverse also holds, not only in $\R^n$ (see Lemma 1 in~\cite{Ko}), but also in the context of the LCA groups.

\medskip

\begin{lem}
Let $G$ be an LCA group and $H\subset G$ a countable discrete subgroup.
A mesurable set $\Omega\subset G, \;k$-tiles $G$ under the translation set  $H$,
if and only if  $$\Omega = \Omega_1\cup\dots\cup \Omega_k\cup R,$$ where $R$ is a zero measure set, 
and the sets $\Omega_j$,  $ 1\leq j\leq k$ are measurable, disjoint and each of them tiles $G$ by translations of  $H$.
\end{lem}

\medskip

\bdem 
If $\Omega$ is a  disjoint union of $k$ sets of representatives of $G/H$ up to measure zero then clearly 
$ \sum_{h\in H} \sub{\chi}{\Omega} (x-h)=k,  \peso{a.e.}$

For the converse, consider $D$ to be a fundamental domain of $G/H$ and let
$\{h_j\}_{j\in\N}$ be an enumeration of the elements of $H.$
 We have  $\Delta_{\Omega}(d)= k$ for almost all $d \in D$. If $E$ denotes the set of the exceptions,
 define for $d\in D\setminus E$,
 $$
i_j(d)= \min\{n\in\N:  \sum_{s=1}^n \sub{\chi}{\Omega} (d+h_s)=j\}, \;\qquad j=1,\dots,k.
$$
and the measurable sets,
$$
E_{j,n} = \{d \in D \setminus E:   i_j(d) = n\}, \;\qquad  n \in \N.
$$
Finally for $j=1,\dots,k,$ let $\;\;\Omega_j = \bigcup_{n\in\N} (E_{j,n}+h_n).$  It is straightforward to see that $\Omega  = \bigcup_{j=1}^k \Omega_j \cup R$,  is the desired decomposition. Here the remaining set $R=\Omega\setminus(\Omega_1\cup\ldots\cup\Omega_k)$ has measure zero because it is contained in $E+H$.
\edem

Going back to the $H$-invariant spaces, let us recall now the following simple but useful proposition, that in the case of LCA groups is a direct consequence 
of Parseval identity and Weil's formula. From now on, and until the end of this section, $D$ will denote a fundamental domain of $\widehat{G}/\La$.

\medskip
  
\begin{pro}\label{shave}
The map $\cT: L^2(G) \to L^2(D,\ell^2(\La))$ defined by
$$
\cT f(\omega) = \{\hat{f} (\omega+\la) \}_{\la \in \La},
$$
is an isometric isomorphism. Moreover, for each element $h\in H$
$$
\cT (\tau_h\,f )(\omega) = e_h(\w)\{\hat{f} (\omega+\la) \}_{\la \in \La},  
$$
for almost every $\w\in D$.
\end{pro}

\medskip

\begin{rem}\label{medible} 
It is not difficult to see that if $\hat{f}$ and $\hat{g}$ are equal almost everywhere, then for almost every $\omega\in D$ 
$$
 \{\hat{f} (\omega+\la) \}_{\la \in \La}= \{\hat{g} (\omega+\la) \}_{\la \in \La}.
$$
This guarantees that $\cT$ is well defined, and justifies the evaluation of elements of $L^2(G)$. 
\EOE
\end{rem}

\medskip

When  the $H$-invariant space is finitely generated, Proposition \ref{shave} allows to translate 
a problem in (infinite dimensional) $H$-invariant spaces,  to simpler linear algebra problems in finite dimensional Hilbert spaces. 

\medskip

Given an $H$-invariant space $V\subset L^2(G)$, there exists a measurable function $J_V$ which is defined in  $D$, takes values on the space of closed subspaces of $\ell^2(\La)$, and has the property that $f\in V$ if and only if for almost every $\w\in D$
$$
\cT f(\w)\in J_V(\w).
$$
The measurability of $J$ is understood in a weak sense, i.e., for every $v,w\in \ell^2(\La)$ the scalar function $\w\mapsto 
\pint{P_{J_V(\w)}v,w}$ is measurable, where $P_{J_V(\w)}$ denotes the orthogonal projection onto $J_V(\w)$. The function $J_V$ was introduced by Helson and it is called \textbf{range function}. Although many range functions can be defined for the same $H$-invariant space V, any two of them coincides almost everywhere (see \cite{CP} for more details).

\medskip

In the case of Paley-Wiener spaces, we have a very special description of a range function. Consider a Borel set $\W\subseteq \G$ such that $m_{\G}(\Omega)<\infty$ and assume that there exists a set $E\subset D$ of zero measure such that $\Delta_{\Omega}$ is uniformly bounded on $D\setminus E$ . If by $J_\W$ we denote the range function of $PW_\W$ then
 it is not difficult to prove the following characterization of $J_\W$.
\begin{pro}
Denote by $C_b(\Omega)$  the set of bounded continuous functions on $\Omega$, extended as zero outside 
$\Omega$, then for each $\w\in D\setminus E$ we have,
\begin{align*}
J_\Omega(\w)
&=\big\{\{g(\w+\la)\}_{\la\in\La}:\  g\in C_b(\Omega)\big\}.
\end{align*}
\end{pro}
Note that, for any continuous function $g$ defined in $\G$, it holds that $g\chi_\Omega \in C_b(\Omega)$. However, the converse it is not necessarily true. From this characterization of $J_\W$ we also get the following one:

\begin{cor}\label{dim fiber}
Given $\w\in D\setminus E$, let $\la_1,\ldots,\la_m$ be the elements of $\La$ such that $\w+\la_j\in \Omega$. Then
$$
J_\W(\w)=\mbox{span}\{\delta_{\la_j}: \ j=1,\ldots,m\}.
$$ 
where $\{\delta_\la\}_{\la\in\La}$ denotes the canonical basis of $\ell^2(\La)$.
\end{cor}
\bdem
Clearly, given $v\in J_\W(\w)$, $v_\la=0$ if $\la\neq \la_j$ for $j=1,\ldots,m$. This proves one of the inclusions.  On the other hand, fix any $j\in\{1,\ldots,m\}$ and take a bounded continuous function $g_{\la_j}$ defined in $\G$ such that $g_{\la_j}(\w+\la_j)=1$ and $g_{\la_j}(\w+\la)=0$ for any other $\la\in \La$. Setting   $g=g_{\la_j}\chi_\W$ we see that $g\in C_b(\Omega)$, and clearly $\{g(\w+\la)\}_{\la\in\La}=\delta_{\la_j}$. This concludes the proof.
\edem

Note that, another consequence of this lemma is that a Borel set $\Omega$ is $k$-tiling if and only if almost every $J_\W(\w)$ are $k$ dimensional.

\medskip

These considerations, as well as Proposition \ref{shave}, lead to the following result.

\begin{teo}\label{ShiftBFR0}
Let $\Omega$ be a $k$-tiling measurable subset of $\widehat{G}$. Given $\phi_1,\ldots,\phi_k \in PW_\Omega$ we define
$$
T_\w=\begin{pmatrix}
\widehat{\phi}_1(\w+\la_{1})&\ldots& \widehat{\phi}_k(\w+\la_{1})\\
\vdots&\ddots&\vdots\\
\widehat{\phi}_1(\w+\la_{k})&\ldots& \widehat{\phi}_k(\w+\la_{k})
\end{pmatrix},
$$ 
where the $\la_j=\la_j(\w)$ for $j=1,\ldots,k$ are the $k$ values of $\la\in\La$ such that $\w+\la\in\Omega$. Then, the following  statements are equivalent:

\begin{itemize}
  \item[(i)] The set $\Phi_H=\{\tau_h\phi_j:\ h\in H\,,\ j=1,\ldots,k\}$ is a Riesz basis for $PW_\Omega$.
    \item[(ii)]  There exist $A, B>0$ such that for almost every $\omega \in D$,
           $$
            A||x||^2 \leq \|T_\w \,x\|^2\leq B||x||^2,
            $$
            for every $x\in \C^k$.
\end{itemize}
Moreover, in this case the constants of the Riesz basis are 
$$
A=\inf_{\w\in D} \ \|T_\w^{-1}\|^{-1} \peso{and} B=\sup_{\w\in D}\ \|T_\w\|.
$$
\end{teo}

For a sake of completeness, we will give a proof of this result adapted to our setting. For the proof in more general $H$-invariant spaces  see \cite{CP}.

\bdem
Recall that $D$ is a fundamental domain of $\widehat{G}/\La$, and consider a family $\{a_{j,h}\}$ with finitely many non-zero terms, where $j=1,\ldots,k$ and $h\in H$. Using the Fourier transform and a $\La$-periodization argument we get 
\begin{align*}
\left\| \sum_{j,h} a_{j,h}\tau_h\phi_j\right\|^2_{L^2(G)}&=\int_D\sum_{j,\ell=1}^{k}m_j(\w)\left(\sum_{\la\in\La} \hat\phi_j(\w+\la)\overline{\hat\phi_\ell(\w+\la)}\right) \overline{m_\ell(\w)}\ dm_{\widehat{G}}(\w). 
\end{align*}
where $m_j=\sum_{h\in H}a_{j,h}e_{-h}$.  For each $j$, the vector $\{\hat{\phi}_j(\w+\la)\}$ has at most $k$ non-zero coordinates. More precisely, the only coordinates that can be different from zero are those corresponding to the elements $\la_j(\w)\in\La$ considered in the matrix $T_\w$. So, if $m=(m_1,\ldots,m_k)$ then
\begin{align}
\left\|\sum_{j=1}^k\sum_{h\in H} a_{j,h}\tau_h\phi_j\right\|^2_{L^2(G)}&=\int_D\pint{T_\w^*T_\w m(\w),m(\w)}_{\C^k}\ dm_{\widehat{G}}(\w)\nonumber\\
&=\int_D\|T_\w m(\w)\|^2_{\C^k}\ dm_{\widehat{G}}(\w).\label{ala1}
\intertext{On the other hand}
\int_D\|m(\w)\|^2_{\C^k}\,dm_{\widehat{G}}(\w)&=\sum_{j=1}^k \int_D|m_j(\w)|^2\,dm_{\widehat{G}}(\w)=\sum_{j=1}^k\sum_{h\in H}|a_{j,h}|^2. \label{ala2}
\end{align}
Combining \eqref{ala1}, \eqref{ala2} and standard arguments of measure theory we get that (i)$\Longrightarrow$(ii). 

\medskip

For the other implication, note that from \eqref{ala1} and \eqref{ala2} we immediately get that the family $\Phi_H$ is a Riesz sequence for $PW_\Omega$. So, it only remains to prove that condition (ii) also implies that $\Phi_H$ is complete. With this aim, let $f\in PW_\Omega$, and suppose that $\pint{f,\tau_h\phi_j}=0$ for every $h\in H$ and $j=1,\ldots,k$.  By a $\La$-periodization argument we get
\begin{align*}
0=\pint{\hat{f},e_{-h}\hat{\phi}_j}&=\int_{D}\left(\sum_{\la\in\La}\hat{f}(\w+\la) \overline{ \hat{\phi}_j(\w+\la)}\right)e_{h}(\w)\ dm_{\widehat{G}}(\w).
\end{align*}
Since $\{e_h\}_{h\in H}$ is an orthonormal basis for $L^2(\widehat{G}/\La)$, then
$$
\sum_{\la\in\La}\hat{f}(\w+\la) \overline{ \hat{\phi}_j(\w+\la)}=0, \ \ a.e. \ \w\ m_{\widehat{G}}.
$$
Thus  $T_\w^*(\{f(\w+\la_j)\}_{j=1}^k)=0$, where $\la_j=\la_j(w)$. Therefore, $\cT(f)=0$ a.e $m_{\widehat{G}}$.
\edem

% ----------------------------------------------------------------
% ----------------------------------------------------------------
\section{Sampling and interpolation near the critical density} \label{critical dens}
% ----------------------------------------------------------------
% ----------------------------------------------------------------

\noi In this section we study sampling sets, and interpolation sets on $PW_{\Omega}$, when $\Omega$ is a relatively compact subset of $\widehat{G}$. In \cite{GKS}, Gr\"ochenig, Kutyniok, and Seip introduced two notions of densities that suitably generalize the Beurling's densities defined in $\R^d$.  
Our main goal is to prove that there exist sampling sets, and interpolation sets whose densities are arbitrarily close to the critical one, answering a question raised by Gr\"ochenig et. al. in \cite{GKS}. 

\subsection{Standing hypothesis} \label{stand}

\noi Since we will work with relatively compact sets $\Omega$, throughout this paper we will assume that $G$ is a second countable LCA group such that its dual $\widehat{G}$ is compactly generated. In order to avoid trivial cases, we will also assume that $G$ is not compact. 

\medskip

By the standard structure theorems, $\widehat{G}$ is isomorphic to $\R^d\times \Z^m\times \K$, where $\K$ is a compact subgroup of $\T^\omega$. Consequently, $G$ is isomorphic to $\R^d\times\T^m\times \D$, where $\D$ is a countable discrete group. This is not a serious restriction, as the following lemma shows (see \cite{FG} or \cite{GKS}). 

\begin{lem} \label{tec}
Assume that $\Omega\subseteq \widehat{G}$ is relatively compact, and let $H$ be the open subgroup generated
by $\Omega$. Then $H$ is compactly generated and there exists a compact subgroup $K\subseteq G$ such
that every $f\in PW_\Omega$ is $K$-periodic. Furthermore, the quotient $G/K$ is homeomorphic to $\R^d\times\T^m\times \D$, where $\D$ is a countable discrete abelian group, and $\widehat{G/K}\simeq H$.
\end{lem}
 
Therefore, given a relatively compact set $\Omega\subseteq \widehat{G}$, this lemma shows that the space $PW_\Omega$ essentially lives in $L^2(G/K)$, and $\widehat{G/K}\simeq H$ is compactly generated.

\subsection{Beurling-type densities in LCA groups} \label{densitygroup}

\noi To begin with, recall that a subset $\La$ of $G$ is called \textbf{uniformly discrete} if there exists an open set $U$ such that the sets $\la+U$ ($\la$ in $\La$) are pairwise disjoints. In some sense, the densities in $\R^d$ compare the concentration of the points of a given discrete set with that of  the integer lattice $\Z^d$. In a topological group, this comparison is done by means of the following relation:

\begin{fed}
Given two uniformly discrete sets $\La$ and $\La'$ and non-negative numbers $\alpha$ and $\alpha'$, we write $\alpha\La\leqm\alpha'\La'$ if for every $\eps>0$ there exists a compact subset $K$ of $G$ such that for every compact subset $L$ we have
$$
(1-\eps)\alpha\,\#(\La\cap L)\leq \alpha'\,\#(\La'\cap (K+L)).
$$ 
\end{fed}

Now, we have to fix a reference lattice in the group $G$. As we mentioned at the beginning of this section, since $\widehat{G}$ is compactly generated, $G$ is isomorphic to $\R^d\times\T^m\times \D$, where $\D$ is a countable discrete group. So, a natural reference lattice is $H_0=\Z^d\times \{e\}\times \D$. Using this reference lattice, and the above transitive relation, we have all what we need to recall the definitions of upper and lower densities.

\begin{fed}\label{densities LCA}
Let $\La$ be a uniformly discrete subset of $G$. The \textbf{lower uniform density} of $\La$ is defined as
$$
\ede^-(\La)=\sup\{\alpha \in \R^+: \alpha H_0\leqm \La\}.
$$
On the other hand, its \textbf{upper uniform density} is 
$$
\ede^+(\La)=\inf\{\alpha \in \R^+: \La\leqm\alpha H_0\}.
$$
\end{fed}

These densities always satisfy that $\ede^-(\La)\leq \ede^+(\La)$, and they are finite. Moreover, it can be shown that the infimum and the supremum are actually a minimum and a maximum. In the case that both densities coincide, we will simply write $\ede(\La)$. It should be also mentioned that  in the case of $\R^d$, these densities coincide with the Beurling's densities when the reference lattice is $\Z^d$.

\medskip

Using these densities, Gr\"ochenig, Kutyniok, and Seip obtained (see \cite{GKS}), the following extension of the classical result of Landau to LCA groups.

\begin{teo}\label{LCA Landau}
Suppose $\La$ is a uniform discrete subset of $G$. Then
\begin{enumerate}
\item[S)]   If $\La$ is a sampling set for $PW_\Omega$, then $\cD^-(\La)\geq m_{\widehat{G}}(\Omega)$; 
\item[I)]    If $\La$ is an interpolation set for $PW_\Omega$, then $\cD^+(\La)\leq m_{\widehat{G}}(\Omega)$.
\end{enumerate}
\end{teo}

A natural question is whether there exist sampling sets and interpolation sets near the critical density. In the case  $G=\R^d$ a positive answer was given by Marzo in  \cite{JM}. The following theorem is our main result, which completely answers this question.

\begin{teo}\label{near critical}
Let $\Omega$ be a compact subset of $\widehat{G}$, and let $\eps>0$. Then, the following statements hold:
\begin{itemize}
\item[(i)] There exists a sampling set $J_{\eps}$ for PW$_\Omega$ such that 
$$
\ede\left(J_{\eps}\right) \leq m_{\widehat{G}}( \Omega )+\eps.
$$
\item[(ii)] If $m_{\G}(\partial\Omega)=0$, then there exists an interpolation set $J^{\eps}$  for PW$_\Omega$ such that 
$$
\ede\left(J^{\eps} \right)\geq m_{\widehat{G}}(\Omega) - \eps.
$$
\end{itemize}
\end{teo}

Although roughly speaking the strategy of the proof will be similar to the one used in~\cite{LS} and \cite{JM} (see also~\cite{K}), in order to pursue this strategy, we will have to overcome several technical issues. This will be done in the following two sections. Finally, in section~\ref{main proof} we will combine the obtained results, and we will provide the proof of Theorem~\ref{near critical}. 

% ----------------------------------------------------------------
% ----------------------------------------------------------------
\section{Constructing Riesz basis in the context of LCA groups} \label{multi}
% ----------------------------------------------------------------
% ----------------------------------------------------------------

\noi The relation between multi-tiling sets and the existence of Riesz bases in the $\R^d$ setting was first pointed out in \cite{GeL} by Grepstad-Lev. More precisely, they proved that a bounded Riemann integrable Borel set $\Omega\subseteq\R^d$ admits a Riesz basis of exponentials if it multi-tiles $\R^d$ with translation set a lattice $\La$.  Later on, Kolountzakis gave in \cite{Ko} a simpler proof of this result in a slightly more general form (see also 
\cite{KN} for a different approach).  Important special cases had been proved by Lyubarskii-Seip in \cite{LS}, and Marzo in \cite{JM},  (see also~\cite{LR} and~\cite{S}). 

\medskip

 One of the main theorems of this section is the following generalization of  Grepstad-Lev's result to the LCA group setting.

\begin{teo}\label{SNK in LCA}
Let $H$ be a uniform lattice of $G$, $\La$ its dual lattice, and $k\in \N$. Then, there exist $a_1,\dots, a_k \in G,$ depending only on the lattice $\La$, such that for any relatively compact Borel subset $\Omega$ of $\widehat{G}$ satisfying
$$
\Delta_{\Omega}(\omega) := \sum_{\la \in \La} \chi_{\Omega} (\omega-\la) =k, \,\,\,  
\peso{a.e.}  \omega \in \widehat{G},
$$
the set 
$$
\{e_{a_j -h} \, \sub{\chi}{\Omega}:\ h\in H\,,\ j=1,\ldots,k\}
$$
is a Riesz basis for $L^2(\Omega)$.
\end{teo}

\medskip

We would like to emphasize that, in the previous theorem, the same set $\{a_1,\ldots,a_k\}$ can be used for any $k$-tiling set $\Omega$. 
We call such a $k$-tuple $(a_1,\ldots,a_k)$  \textbf{$H$-universal}. The following result is a slight improvement of the already known results.

\begin{teo}\label{full measure}
Let $H$ be a uniform lattice of $G$ and $k\in \N$.  Then, there exists a Borel set $N\subseteq G^k$ such that $m_{G^k}(N)=0$  and every $k$-tuple $(a_1, \ldots,a_k)\in G^k\setminus N$ is $H$-universal.  
\end{teo} 

\medskip

\begin{rem}\label{probability rem}
Note that, if we fix a fundamental domain $D$, given any universal $k$-tuple $(a_1,\ldots,a_k)$ there exists a unique $k$-tuple $(d_1,\ldots,d_k)\in D^k$ such that
$$
\{e_{a_j -h} \, \sub{\chi}{\Omega}:\ h\in H\,,\ j=1,\ldots,k\}=
\{e_{d_j -h} \, \sub{\chi}{\Omega}:\ h\in H\,,\ j=1,\ldots,k\}.
$$
So, we can restrict out attention to universal $k$-tuples belonging to $D^k$.  In this case, consider the ``uniform'' probability measure on $D^k$ given by the restriction of the Haar measure of $G^k$ to $D^k$ (conveniently normalized). Then another way to state Theorem~\ref{full measure} is that a $k$-tuple in $D^k$ is almost surely $H$-universal. \EOE
\end{rem} 

The second main result of this section is the following  converse of Theorem~\ref{SNK in LCA}.

\begin{teo}\label{la vuelta}
Let $H$ be a uniform lattice of $G$ and $\La$ its dual lattice. Given a relatively compact subset $\Omega$ of $\widehat{G}$, if $L^2(\Omega)$ admits a Riesz basis of the form
$$
\{e_{a_j -h} \, \sub{\chi}{\Omega}:\ h\in H\,,\ j=1,\ldots,k\}
$$
for some  $a_1,\dots, a_k \in G,$ then $\Omega $ $k$-tiles $\widehat{G}$ with $\La$.
\end{teo}

The proofs of these results are provided in the next subsection. Then, in the last subsection we will show a counterexample that answer negatively a question raised by Kolountzakis in \cite{Ko}.

\subsection{Proofs of Theorems \ref{SNK in LCA}, \ref{full measure} and \ref{la vuelta}}

\noi We begin with two technical lemmas. Following Rudin's book \cite{R}, we will say that a function $p$ is a \textbf{trigonometric polynomial} on $G$ if it has the form
$$
p(g)= \sum_{j=0}^n c_j \gamma_{j} (g)
$$
for some $n\in\N$, $c_j\in\C$ and $\gamma_j\in \widehat{G}$.

\begin{lem} \label{nonzero}
The zero set of a trigonometric polynomial $p$ on $G$ has zero Haar measure.
\end{lem}

\bdem

By the standing hypothesis, we can identify $G$ with the group $\R^d\times\T^m\times \D$, for some  countable discrete LCA group $\D$. Hence, given $(x,\w,d)\in \R^d\times\T^m\times \D$, the polynomial $p$ can be written as
$$
p(x,\w,d)= \sum_{j=0}^n c_j \rho_j(x)\tau(\w)\delta(d)
$$
where $\rho\in \widehat{\R}^d$, $\tau\in \widehat{\T}^d$, and $\delta\in \widehat{\D}$. Let $C_p= \{(x,\w,d): p(x,\w,d)=0\}$, and suppose by contradiction that $m_G(C_p)>0$. Since
$$
C_p=\bigcup_{d\in D} C_p\cap\big( \R^d\times\T^m\times \{d\} \big),
$$
there exists $d_0\in D$ such that
$$
m_G\big(C_p\cap \big(\R^d\times\T^m\times \{d_0\}\big)\big)>0
$$
If we restrict $p$ to $\R^d\times\T^m\times \{d_0\}$ we get the trigonometric polynomial $q$ on $\R^d\times \T^m$ 
$$
q(x,\w)= \sum_{j=0}^n \big(c_j \delta(d)\big)\, \rho_j(x)\tau(\w)
$$
that is non-trivial and its zero set has positive measure. This is a contradiction, and therefore $m_g(C_p)=0$.
\edem

\begin{lem}\label{ulala}
Let $K_1$ and $K_2$ be compact subsets of $\widehat{G}$. If
$$
\Gamma=\{\la\in\La:\ (\la+K_1)\cap K_2\neq \varnothing\},
$$
then $\#\Gamma<\infty$.
\end{lem}

\bdem
Note that $\Gamma \subset \La \cap (K_1- K_2)$, where $K_1-K_2=\{k_1-k_2:  k_j\in K_j, j=1,2\}$. Since $\La$ is a discrete set and $(K_1 - K_2)$ is compact, then $\Gamma$ should be necessarily a finite set. 
\edem

\medskip

\begin{proof}[Proof of Theorems \ref{SNK in LCA} and \ref{full measure}]
Given $a_1,\ldots,a_k\in G$, define the functions $\phi_1,\ldots,\phi_k$ by their Fourier transform in the following way:

 \begin{equation} \label{fies}
 \hat{\phi}_j:= e_{a_j} \, \sub{\chi}{\Omega}, \quad  j\in \{1, \ldots, k\}.
 \end{equation}
 We will show that under the hypothesis on $\Omega$, there exist  $a_1, \dots, a_k$ such that $\phi_1,\ldots,\phi_k$  translated by $H$ form a Riesz basis for $PW_\Omega$. 

\medskip

Choose a fundamental domain $D$ of $\widehat{G}/\La$.  Since $\Omega$ is a set that $k$-tiles $\widehat{G}$, for almost every $\w\in D$, the vectors $\widehat{\phi}_j(\w)$ have at most $k$ entries different from zero. These entries are those that correspond to the (different) elements $\la_{j}=\la_j(\w)\in\La$, $1\leq j \leq k$, such that $\w+\la_{j}\in\Omega$. For $\w\in D$  consider  the matrix
$$
T_\w=\begin{pmatrix}
\widehat{\phi}_1(\w+\la_{1})&\ldots& \widehat{\phi}_k(\w+\la_{1})\\
\vdots&\ddots&\vdots\\
\widehat{\phi}_1(\w+\la_{k})&\ldots& \widehat{\phi}_k(\w+\la_{k})
\end{pmatrix}.
$$
By Theorem \ref{ShiftBFR0}, the $H$-translations of  $\phi_1,\ldots,\phi_k$  form a Riesz basis for $PW_\Omega$ if and only if there exist $A, B>0$ such that
\begin{equation}\label{all bounds}
     A||x||^2 \leq \|T_\w x\|^2 \leq B||x||^2,
\end{equation}
for every $x\in \C^k$ and almost every $\w\in D$. The rest of the proof follows ideas of~\cite{Ko} suitably adapted to our setting. Firstly, note that
\begin{align}
T_\w&=\begin{pmatrix}
\widehat{\phi}_1(\w+\la_{1})&\ldots& \widehat{\phi}_k(\w+\la_{1})\\
\vdots&\ddots&\vdots\\
\widehat{\phi}_1(\w+\la_{k})&\ldots& \widehat{\phi}_k(\w+\la_{k})
\end{pmatrix}=\begin{pmatrix}
 e_{a_1}\, (\w+\la_1)     &\ldots&       e_{a_k}\, (\w+\la_1)  \\
\vdots&\ddots&\vdots\\
 e_{a_1}\, (\w+\la_k)       &\ldots&    e_{a_k}(\w+\la_k) 
\end{pmatrix}\nonumber \\&=
\begin{pmatrix}
 (a_1,\la_1) &\ldots& (a_k,\la_1)   \\
\vdots&\ddots&\vdots\\
(a_1,\la_{k})&\ldots& (a_k,\la_{k})
\end{pmatrix}\begin{pmatrix}
(a_1,\w)   &0&\ldots&0& 0\\
0&(a_2,\w) &\ldots& 0&0\\
\vdots&\vdots&\ddots&\vdots&\vdots\\
0&0&\ldots& (a_{k-1}, \w)&0\\
0&0&\ldots& 0&(a_k,\w)
\end{pmatrix}  = E_\w\, U_\w \,. \label{short}
\end{align}

Since $U_\w$ is unitary, to prove the inequalities in \eqref{all bounds} is equivalent to show that
\begin{equation}\label{all bounds modified}
     A||x||^2 \leq \|E_{\omega} x\|^2 \leq B||x||^2,
\end{equation}
for every $x\in\C^k$, and almost every $\w\in D$. By Lemma \ref{ulala}, applied to $K_1=\overline{D}$ and $K_2=\overline{\Omega}$, when $\w$ runs over (a full measure subset of) the fundamental domain $D$, only a finite number of different matrices $E_\w$ appear in \eqref{short},  say $E_1,\ldots, E_N$. Thus, it is enough to prove that they are all invertible. Note that the determinants of the $E_\w$ are polynomials of the form
$$
d(x_1,\ldots,x_k)=\sum_{\pi\in S_k}\sgn(\pi) \prod_{j=1}^k(x_{\pi(j)},\la_j(\w)) ,
$$
evaluated in $(a_1,\ldots,a_k) \in G\times \dots \times G= G^k$, where $S_k$ denotes the permutation group on $1,\dots, k$.  
Since $\La$ is countable, the set of trigonometric polynomials on $G^k$
$$
\cP_k=\left\{p(x_1,\ldots,x_k)=\sum_{\pi\in S_k}\sgn(\pi) \prod_{j=1}^k(x_{\pi(j)},\la_j): \mbox{for any $(\la_1,\ldots,\la_k)\in \La^k$} \,\right\}
$$
is countable. This set contains the trigonometric polynomials $d(x_1,\ldots,x_k)$ associated to the determinants of the matrices $E_j$. Note that it also contains the polynomials associated to matrices $E_j'$ coming from any other $k$-tiling set. Therefore, the universal $k$-tuple $(a_1,\ldots,a_k)$ that we are looking for, is any $k$-tuple such that
$$
p(a_1,\ldots,a_k)\neq 0 \quad \forall p\in\cP_k.
$$
To prove that such a $k$-tuple exists, we will use a measure theoretical argument based on Lemma \ref{nonzero}. Note that $G^k$ is a compactly generated LCA group, and $\La^k$ is the dual lattice of the the uniform lattice $H^k$ in $G^k$. Hence, using Lemma \ref{nonzero} with $G^k$ instead of $G$, we get that the union of the zero sets corresponding to these polynomials has zero Haar measure in $G^k$. Therefore, there exist $a_1,\dots, a_N \in G$ so that $(a_1,\dots, a_N)$ does not belong to any of these zero sets. In particular, for these values of $a_j$, the matrices $E_1,\ldots,E_N$ are invertible. Then, by Theorem \ref{ShiftBFR0}, the $H$-translations of the functions $\phi_1,\ldots,\phi_k$  form a Riesz basis for $PW_\Omega$. This is equivalent to say that
 $$
\{e_{a_j -h} \, \sub{\chi}{\Omega}:\ h\in H\,,\ j=1,\ldots,k\},
 $$
 is a Riesz basis on $L^2(\Omega)$. The same holds for any other $k$-tiling set $\Omega'$ by construction of $\cP_k$ and the $k$-tuple $(a_1,\ldots,a_k)$.
\edem

\begin{proof}[Proof of Theorem \ref{la vuelta}]
Recall that for each $\w\in D$
$$
J_\Omega(\w)=\big\{\{g(\w+\la)\}_{\la\in\La}:\ g\in C_b(\Omega) \big\}. 
$$
The hypothesis implies that there exists $E\subset D$ of measure zero such that, given $w\in D\setminus E$,  the set of vectors
\begin{equation}\label{eq riesz basis}
\{\{e_{a_j}(\w+\la)\chi_\Omega(\w+\la)\}_{\la\in\La}:\ j=1,\ldots,k\}
\end{equation}
is a Riesz basis of $J_\Omega(\w)$. In particular, $\dim J_\Omega(\w)=k$ for all $\w\in D\setminus E$. By Corollary \ref{dim fiber}, if 
$\la_1,\ldots,\la_m$ be the elements of $\La$ such that $\w+\la_j\in \Omega$ then
$$
J_\W(\w)=\mbox{span}\{\delta_{\la_j}: \ j=1,\ldots,m\}.
$$ 
where $\{\delta_\la\}_{\la\in\La}$ denotes the canonical basis of $\ell^2(\La)$. This implies that
$$
\#\{\la\in\La:\ \w+\la\in\Omega\}=\dim J_\W(\w),
$$
and we already know that this dimension is equal to $k$ for every $\w\in D\setminus E$. Thus, $\W$ is $k$-tiling with respect to the lattice $\La$.
\end{proof}

\subsection{A counterexample for unbounded sets in $\R$} \label{counterex}

\noi The same scheme can not be applied if $\Omega$ is not relatively compact, as the following example shows. This example also gives a negative answer to the open problem left by Kolountzakis in \cite{Ko}.

\begin{exa}\label{ejemplito}
Let $G=\R$, and consider the following subset of $\widehat{\R}\simeq\R$:
$$
\Omega_0=[0,1)\cup \bigcup_{n=2}^\infty[n-2^{-(n-2)},n-2^{-(n-1)}).
$$

\noi This is a $2$-tiling set with respect to the lattice $\Z$ (see Figure \ref{contraejemplo}).

\bigskip

\begin{figure}[H]
           \centering
           \includegraphics[width=\textwidth]{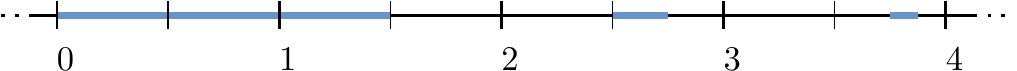}
           \par\vspace{0cm}
           \caption{The set $\Omega_0$.}
           \label{contraejemplo}
\end{figure}

 However, if we consider the functions $\phi_1$ and $\phi_2$ defined through their Fourier transform by 
$$
\hat{\phi}_j(\w)=e^{2\pi i a_j \w} \, \sub{\chi}{\Omega_0},
$$
for $j=1,2$, then integer translations of $\phi_1$ and $\phi_2$ are not a Riesz basis for $PW_\Omega$ for any choice of $a_1,a_2\in\R$. In other words, $(a_1+\Z) \cup (a_2+\Z)$ is not a complete interpolation set for $PW_\Omega$ for any pair $a_1,a_2\in\R$. 

\medskip

To show this, recall that in the proof of  Theorem \ref{SNK in LCA} we proved that the integer translations of $\phi_1$ and $\phi_2$ form a Riesz basis for $PW_\Omega$ if and only if the matrices
\begin{align*}
E_\w&=\begin{pmatrix}
e^{2\pi i a_1 \la_1(\w) }&e^{2\pi i a_2 \la_1(\w) }\\
e^{2\pi i a_1 \la_2(\w) }&e^{2\pi i a_2 \la_2(\w) }
\end{pmatrix},
\end{align*}
and their inverses are uniformly bounded for almost every $\w$ in the fundamental domain, which we for simplicity  choose to be the interval $[0,1)$. For this particular $\Omega_0$, $\la_1(\w)$ is always equal to zero, while  $\la_2(\w)=n$ if $\w\in[1-2^{-(n-1)},1-2^{-n})$, for $n\in\N$. Therefore 
\begin{align*}
E_\w=\begin{pmatrix}
1&1\\
e^{2\pi i a_1 \la_2(\w) }&e^{2\pi i a_2 \la_2(\w) }
\end{pmatrix},
\end{align*}
which can be rewritten as
\begin{align*}
E_\w=\begin{pmatrix}
1&0\\
0&e^{2\pi i a_1 \la_2(\w) }
\end{pmatrix}
\begin{pmatrix}
1&1\\
1&e^{2\pi i (a_2-a_1) \la_2(\w) }
\end{pmatrix}.
\end{align*}
So, if $a_2-a_1\in \Q$, there exists a set of positive measure such that the matrices $T_\w$ are not invertible for $\w$ in that set. On the other hand, if $a_2-a_1\notin\Q$, as the set $\{e^{2\pi i (a_2-a_1) n }\}_{n\in\N}$ is dense in $\T$,  the matrices $E_\w^{-1}$ are not uniformly bounded.

\medskip

Since the set $\Omega_0$ multi-tiles $\R$ with other lattices too, a natural question is whether or not we can obtain Riesz basis using these lattices. The answer is no, and the idea of the proof is essentially the same. For this reason, we only make some comments on the main differences, and we leave the details to the reader. First of all, recall that the (uniform) lattices of $\R$ have the form $\La_\alpha=\alpha \Z$, for some $\alpha\in\R$. It is not difficult to prove that $\Omega_0$ multi-tiles $\R$ only for lattices corresponding to $\alpha=k^{-1}$, for some $k\in\N$.  Moreover, with respect to the lattice $\La_{k^{-1}}$, the set $\Omega_0$ is $2k$-tiling. Furthermore, in the proof of Theorem \ref{SNK in LCA} we showed that  given $a_1,\ldots, a_{2k}\in\R$ then we have: $(a_1+\Z)\cup \ldots\cup (a_{2k}+\Z)$ is a complete interpolation set for $PW_\Omega$ if and only if the matrices
$$
E_\w=\begin{pmatrix}
1&1&\ldots&1&1\\
e^{2\pi i a_1\la_1(\w) }&e^{2\pi i a_2\la_1(\w) }&\ldots & e^{2\pi i a_{2k-1}\la_1(\w) }&e^{2\pi i a_{2k}\la_1(\w)}\\
\vdots                    & \vdots                   &\ddots & \vdots                         &\vdots\\
e^{2\pi i a_1\la_{2k-1}(\w)}&e^{2\pi i a_2\la_{2k-1}(\w) }&\ldots&e^{2\pi i a_{2k-1}\la_{2k-1}(\w)}&e^{2\pi i a_{2k}\la_{2k-1}(\w)}\\
e^{2\pi i a_1\la_{2k}(\w)}&e^{2\pi i a_2\la_{2k}(\w) }&\ldots&e^{2\pi i a_{2k-1}\la_{2k}(\w)}&e^{2\pi i a_{2k}\la_{2k}(\w)}
\end{pmatrix}
$$
and their inverses are uniformly bounded for almost every $\w\in[0,k^{-1})$.  By construction of $\Omega_0$, for each $m\in\N$ we can find an interval $I\subseteq [0,k^{-1})$ of positive measure such that for every $\w\in I$ there exists $j\in\{1,\ldots,2k\}$ so that $\la_j(w)=m$. On the other hand, as in the case studied before, either the orbit 
$$
\{(e^{2\pi i a_1\, m },e^{2\pi i a_2\,m},\ldots , e^{2\pi i a_{2k-1}\,m},e^{2\pi i a_{2k}\,m})\}_{m\in\Z},
$$
is periodic or there exists elements of the orbit as close as we want to the first row vector $(1,\ldots,1)$. Therefore, the matrices $E_\w^{-1}$ can not be uniformly bounded in a full measure subset of $[0,1)$.
\end{exa}

\section{Quasi-dyadic cubes} \label{Dyadics}

\noi In the previous section we proved that there exist many Borel sets $\Omega$ in $\widehat{G}$ such that $L^2(\Omega)$ admits a Riesz basis of characters. However, in general, there are not enough sets of this kind so that they can be used to approximate an arbitrary compact set. Let us briefly explain the reason. In the classical case of $\R^d$, the approximation is done by means of sets that are union of dyadic cubes. Note that the dyadic cubes of side length equal to $2^{-n}$ are fundamental domains for the lattice $2^{-n}\Z$. Hence, the dilation of the cubes is reflected in the refinement of the lattices.  So, in order to get a good approximation in more general LCA groups, the idea is to look for a nested family of lattices whose corresponding fundamental domains become in some sense smaller and smaller. This is the main issue, because to get such family in the group $\G\simeq\R^d\times \Z^m\times \K$ may be difficult (or even impossible), because of the compact factor. The key to overcome this difficulty is given by the following classical result (see \cite{HR1}).

\begin{lem}\label{grupo chico}
Given a neighborhood $U$ of $e$ in $\widehat{G}$, there exists a compact subgroup $K$ included in $U$ and such that $\widehat{G}/K$ is elemental, that is
\begin{equation}\label{ident}
\widehat{G}/K\simeq \R^d\times \Z^m\times \T^\ell\times F,
\end{equation}
where $F$ is a finite group. 
\end{lem}

In $\R^d\times \Z^m\times \T^\ell\times F$, the above mentioned strategy to obtain dilations by means of refinements of lattices can be done without any problem. More precisely, we consider 
\begin{equation}\label{losene}
\begin{array}{l}
\La_n\ \ =\, \La_n(d,m,\ell)\, =\, (2^{-n}\Z)^d\times\Z^m\times\Z_{2^n}^\ell \times F\, \subseteq \widehat{G}/K\,,\\
\ddd_{0}^{(n)}=[-2^{-n-1},2^{-n-1})^d\times \{0\}\times[-2^{-n-1},2^{-n-1})^\ell\times \{e\}\,.\\
\end{array}
\end{equation}

This leads to the following definition of quasi-dyadic cubes.

\begin{fed}\label{qdc}
Let $K$ be a compact subgroup of $\widehat{G}$ such that $\widehat{G}/K$ is elemental, and let $\pi$ the canonical projection from $\widehat{G}$ onto the quotient. Identifying the quotient $\widehat{G}/K$ with the group $\R^d\times \Z^m\times \T^\ell\times F$,  the family of \textbf{quasi-dyadic cubes of generation $n$ associated to $K$}, denoted by $\dcg{K}{n}$, are defined by
$$
Q_{\la}^{(n)}=\pi^{-1}(\cQ_{\la}^{(n)})
$$ 
where $\cQ_\la^{(n)}=\la+\cQ_0^{(n)}$ for $\la\in\La_n$.
\end{fed}

Note that in order to distinguish the cubes in the quotient from the cubes in $\widehat{G}$, for those in $\widehat{G}/K$ we use calligraphic letters. Note also that the quasi-dyadic cubes $Q_{\la}^{(n)}$ are relatively compact. Indeed, if $S_\la^{(n)}$ is a relatively compact Borel section of $\cQ_\la^{(n)}$ in the group $\widehat{G}$, then $Q_{\la}^{(n)}=S_{\la}^{(n)}+K$.

\medskip

The main difference with the classical case, is that the quasi-dyadic cubes are para\-metrized not only by a (dyadic) lattice, but also by some compact subgroups. This family of quasi-dyadic cubes clearly satisfies many of the arithmetical and combinatorial properties of the classical dyadic cubes. However, for our purposes, the following approximation result is the most important.

\begin{pro}\label{tecnico1}
Let  $C$ be a compact set and $V$ an open set such that $C\subset V\subset \widehat{G}$. There exists a compact subgroup $K$ of $\widehat{G}$ such that $\widehat{G}/K$ is an elemental LCA group, and $Q_{\la_1}^{(m)},\ldots, Q_{\la_k}^{(m)}\in \dcg{K}{m}$ for $m\in\N$ large enough  such that
$$
C\subseteq \bigcup_{j=1}^k Q_{\la_j}^{(m)}=\pi^{-1}\left(\bigcup_{j=1}^k  \la_j+\cQ_{0}^{(m)}\right)\subseteq V.
$$
\end{pro}

\bdem
Let $U$ be a compact neighborhood of $e$ in $\widehat{G}$ such that  $C+U\subseteq V$. Take a compact subgroup $K$ contained in $U$ which satisfies
$$
\widehat{G}/K\simeq \R^d\times \Z^m\times \T^\ell\times F,
$$
for some integers $d,m,\ell\geq 0$. Let $\pi:\widehat{G}\to\widehat{G}/K$ denote the canonical projection. By our assumptions on the open set $U$ we have for $n$ large enough that
$$
C\subseteq C+{Q}_0^{(2n)}\subseteq C+{Q}_0^{(n)}\subseteq V.
$$ 
On the other hand, by the compactness of $C$, there exist $\gamma_1,\ldots,\gamma_j\in C$ such that
$$
C\subseteq \bigcup_{i=1}^j \big(\gamma_i+Q_0^{(2n)}\big).
$$
Consider the lattice $\La_{4n}$, and $\la_{i,1},\ldots,\la_{i,s_i}\in\La_{4n}$ such that
\begin{align*}
\pi(\gamma_i+Q_0^{(2n)})&=\pi(\gamma_i)+\ddd_0^{(2n)}\subseteq \bigcup_{h=1}^{s_i} \big(\la_{i,h}+\ddd_0^{(4n)}\big)\subseteq \pi(\gamma_i) +\ddd_0^{(n)}.
\end{align*}
Let $\{\la_1,\ldots,\la_k\}$ an enumeration of the elements of $\La_{4n}$ used to cover all the sets $\pi(\gamma_i)+\ddd_0^{(2n)}$. Then, the above inclusions imply that
\begin{align*}
C&\subseteq \bigcup_{i=1}^{j} \big(\gamma_i +{Q}_0^{(2n)}\big)\subseteq \bigcup_{j=1}^k \pi^{-1}\big(\la_j+\ddd_0^{(4n)}\big)\subseteq \bigcup_{i=1}^{j} \big(\gamma_{i}+{Q}_0^{(n)}\big)\subset V
\end{align*}
Thus, we can take $m=4n$, and the proof is complete.
\edem

\medskip

 Another good property of the quasi-dyadic cubes is the following.

\begin{pro}\label{Riesz basis in union of qdc}
Let $\Omega$ be finite a union of quasi-dyadic cubes in $\dcg{K}{n}$. Then, the space $L^2(\Omega)$ admits a Riesz basis of characters (restricted to $\Omega$).
\end{pro} 

If $\pi:\G\to\G/K$ denotes the canonical projection, then the set $\pi(\Omega)$ multi-tiles the quotient space with the lattice $\La_n$ defined in \eqref{losene}. By Theorem \ref{SNK in LCA}, this implies the existence of a Riesz basis of characters in the space $L^2(\pi(\Omega))$. Now, we need a result that gives us a way to lift this basis. This is provided by the following result, which in particular concludes the proof of Proposition \ref{Riesz basis in union of qdc}.

\begin{teo}\label{tensor product}
Let $K$ be a compact subgroup of an LCA group $G$ such that $\widehat{K}$ is countable. Suppose that there exists a subset $Q$ of $G/K$ such that $L^2(Q)$ admits a Riesz basis of characters of $G/K$. If $\pi:G\to G/K$ denotes the canonical projection, and $\widetilde{Q}=\pi^{-1}(Q)$, then $L^2(\widetilde{Q})$ also admits a Riesz basis of characters.  
\end{teo}

\bdem
On the one hand, note that $\widehat{G/K}\simeq K^\bot\subseteq \widehat{G}$, where $K^{\bot}$ denotes the annihilator of $K$. So, the Riesz basis for $L^2(Q)$ can be identified with some elements $\{\gamma_n\}$ in $\widehat{G}$. On the other hand, the elements  $\widehat{K}$ form an orthonormal basis for $L^2(K)$ endowed with the normalized Haar measure $m_K$. Moreover, since $\widehat{K}$ can be identified with the quotient group $\widehat{G}/K^\bot$, the orthonormal basis for $L^2(K)$ can be identified with a system of representatives  $\{\kappa_m\}$ of $\widehat{G}/K^\bot$. 

\medskip

Now, we will prove that $\{{\gamma_n+\kappa_m}\}$ is a Riesz basis for $L^2(\widetilde{Q})$. First of all, we will prove that it is complete. Let $F\in L^2(\widetilde{Q})$ such that $\pint{F,{\gamma_n+\kappa_m}}=0$ for every $n$ and $m$. By the Weil's formula, $m_{G}=m_{K}\times m_{G/K}$ provided we renormalize conveniently the Haar measure on $G/K$. So, using this formula and the fact that $(k,{\gamma_n})=1$ for every $k\in K$ we get for every $m$ and every $n$
\begin{align*}
0&=\int_{\widetilde{Q}} F(g) \overline{(g,{\gamma_n+\kappa_m})} \, dm_{G}(g)\\
&=\int_{Q}\,\left(\int_K F(g+k) \overline{(g+k,{\gamma_n+\kappa_m})} \, dm_{K}(k)\right)\,dm_{G/K}(\pi(g))\\
&=\int_{Q}\,\left(\int_K F(g+k) \overline{(g,{\kappa_m})}\ \overline{(k,{\kappa_m})} \, dm_{K}(k)\right)(\pi(g),{\gamma_n})\,dm_{G/K}(\pi(g)). 
\end{align*}
Fix $m$. Then, using that $\{{\gamma_n}\}$ is a Riesz basis for $L^2(Q)$ we get that
\begin{align*}
\overline{(g,{\kappa_m})}\,\int_K F(g+k)\  \overline{(k,{\kappa_m})} \, dm_{K}(k)=0\quad \quad m_{G/K} - a.e.
\end{align*}
So, since $\{{\kappa_m}\}$ is a (countable) orthonormal basis of $K$, we get that 
\begin{align*}
\int_K |F(g+k)|^2 \, dm_{K}(k)=\sum_m \left|\int_K F(g+k)\  \overline{(k,{\kappa_m})}\, dm_{K}(k)\right|^2=0 \quad \quad m_{G/K} - a.e.
\end{align*}
So, by the Weil's formula we get:
$$
\|F\|^2_{L^2(\widetilde{Q})}=\int_{Q}\,\left(\int_K |F(g+k)|^2 \, dm_{K}(k)\right)\,dm_{G/K}(\pi(g))=0.
$$
Therefore  $\{{\gamma_n+\kappa_m}\}$ is complete. 

Now, in order to prove that it is also a Riesz sequence, consider a sequence $\{c_{n,m}\}$ with  finitely many non-zero terms. Then
\begin{align}
\left\| \sum_{n,m} c_{n,m} ({\gamma_n+\kappa_m})\right\|^2_{L^2(\widetilde{Q})}&=\int_{\widetilde{Q}}\left | \sum_{n,m} c_{n,m} (g,{\gamma_n+\kappa_m})
\right|^2dm_{G}(g) \nonumber \\
&=\int_{Q}\int_K \left | \sum_{n,m} c_{n,m} (g+k,{\gamma_n+\kappa_m})\right|^2  dm_{K}(k)\,dm_{G/K}(\pi(g)).\label{x}
\end{align}
Since $(k, {\gamma_n})=1$ for every $k\in K$, the sum inside the integrals can be rewritten as
\begin{align*}
\sum_{n,m} c_{n,m}\, (g+k,{\gamma_n+\kappa_m})=\sum_m \left((g,{\kappa_m})\sum_{n} c_{n,m}\, (\pi(g),{\gamma_n}) \right)(k, {\kappa_m}).
\end{align*}
Therefore, using that $\{{k_m}\}$ is an orthonormal basis of $L^2(K)$ we get
\begin{align}\label{xx}
\int_K \left | \sum_{n,m} c_{n,m}\, (g+k, {\gamma_n+\kappa_m})\right|^2 \, dm_{K}(k)
&=\sum_m \left|\sum_{n} c_{n,m}\, (\pi(g),{\gamma_n})\right|^2\,.
\end{align}
So, substituting \eqref{xx} in \eqref{x}, we obtain,
\begin{align*}
\left\| \sum_{n,m} c_{n,m}\, ({\gamma_n+\kappa_m})\right\|^2_{L^2(\widetilde{Q})}
&=\sum_m \int_{Q}\,  \left|\sum_{n} c_{n,m}\, (\pi(g),{\gamma_n})\right|^2\,dm_{G/K}(\pi(g)).
\end{align*}
Finally, since $\{{\gamma_n}\}$ as a Riesz basis for $L^2(Q)$, there exist $A,B>0$ such that 
\begin{align*}
A \sum_{m,n} |c_{n,m}|^2\leq \left\| \sum_{n,m} c_{n,m}\, ({\gamma_n+\kappa_m})\right\|_{L^2(\widetilde{Q})}^2\leq B \sum_{m,n} |c_{n,m}|^2,
\end{align*}
and this concludes the proof.
\edem

\begin{rem}
Statements analogous to Theorem \ref{tensor product} hold when we instead of  Riesz bases consider orthogonal bases or frames of 
characters. A proof in the case of orthonormal basis is contained in the proof of Lemma 3  of \cite{GKS}. \EOE
\end{rem}

\section{Proof of the main result} \label{main proof}

\noi Finally, using the techniques developed in the previous sections, we provide the proof of our main Theorem \ref{near critical}.

\begin{description}
\item[(i) Sampling case:]  Since the Haar measure is regular, there exists an open subset $V$ of $\widehat{G}$ such that $\Omega\subseteq V$ and $m_{\widehat{G}}(V\setminus \Omega)\leq\eps$. By Lemma \ref{tecnico1}, there exists a compact subgroup $K$ of $\widehat{G}$ so that $\widehat{G}/K$ is elemental,  $m\in\N$ large enough, and $Q_{\la_1},\ldots, Q_{\la_k}\in \dcg{K}{m}$ are such that
$$
\Omega \subseteq \bigcup_{j=1}^k Q_{\la_j}^{(m)}\subseteq V.
$$
Let $\Omega_\eps$ be the union of these $k$ quasi-dyadic cubes. Then, by Proposition \ref{Riesz basis in union of qdc}, the space $L^2(\Omega_\eps)$ admits a Riesz basis consisting of characters of $\widehat{G}$ (restricted to $\Omega_\eps$). Let $\{e_{b_n}\,\sub{\chi}{\Omega_\eps}\}$ denote this basis, and let $J_\eps=\{b_n\}\subseteq G$. Using Theorem \ref{LCA Landau} we get that $\cD(J_\eps)=m_{\widehat{G}}(\Omega_\eps)\leq m_{\widehat{G}}(\Omega)+\eps$. Note that $\{e_{b_n} \,\sub{\chi}{\Omega} \}$ is a frame for $L^2(\Omega)$, because it is obtained by projecting a Riesz basis for the bigger space $L^2(\Omega_\eps)$. So, $J_\eps$ is a sampling set for $PW_\Omega$.

\medskip

\item[(ii) Interpolation case:]  Since $m_{G}(\partial\Omega)=0$, we can work with the interior of $\Omega$. For the sake of simplicity we will use the same letter for it. Let $C$ be a compact subset of $\Omega$ such that $m_{\widehat{G}}(\Omega \setminus C)\leq \eps$. Again by Lemma \ref{tecnico1}, there exists a compact subgroup $K$ of $\widehat{G}$ so that $\widehat{G}/K$ is elemental,  $m\in\N$ large enough, and $Q_{\la_1},\ldots, Q_{\la_k}\in \dcg{K}{m}$ are such that
$$
C \subseteq \bigcup_{j=1}^k Q_{\la_j}^{(m)}\subseteq \Omega.
$$
As before, if $\Omega_\eps$ is the union of the $k$ quasi-dyadic cubes, then the space $L^2(\Omega_\eps)$ admits a Riesz basis consisting characters of $\widehat{G}$ (restricted to $\Omega_\eps$). In this case, the set $J^\eps$ consisting of these characters, forms a Riesz sequence in $L^2(\Omega)$. This is equivalent to say that, as points of $G$, they form an interpolation set for $PW_\Omega$. Since $\cD(J^\eps)=m_{\widehat{G}}(\Omega_\eps)\geq m_{\widehat{G}}(\Omega)-\eps$, which concludes the proof.\hfill $\blacksquare$
\end{description}

\vspace{.8cm}

\begin{center}
\textbf{Acknowledgments}
\end{center}

We are indebted to Jordi Marzo for enlightening discussions and, specially we would like to thank him for letting us know the open question raised in \cite{GKS}. We would also like to thank Ursula Molter for helpful discussions that let us improve the presentation of this paper. 
Finally our thanks go to the referees whose comments had made the article more readable.

% ------------------------------------------------------------------------------------------------------------------------------------------------------------------------------------------------------
% ------------------------------------------------------------------------------------------------------------------------------------------------------------------------------------------------------
\end{document}